\documentclass{article}

\usepackage[a4paper, total={7in, 8in}]{geometry}
\usepackage[utf8]{inputenc}   
\usepackage[T1]{fontenc}      
\usepackage{multicol}
\usepackage{amsmath,amsthm} 
\usepackage{amssymb,mathrsfs} 
\usepackage{amssymb}
\usepackage{amsfonts}
\usepackage[normalem]{ulem}

\usepackage{graphicx}
\usepackage{enumerate}
\usepackage{graphicx} 
\usepackage{lmodern} 
\usepackage{mathtools}
\usepackage{dsfont}
\usepackage{physics}
\usepackage{subfig}
\usepackage{framed}
\usepackage{color}

\newcommand\N{\mathbb{N}}
\newcommand\E{\mathbb{E}}

\newcommand\R{\mathbb{R}}

\newcommand\cZ{\mathcal{Z}}

\newcommand\proba{\mathbb{P}}
\newcommand\e{\mathrm{e}}
\renewcommand\liminf{\underline{\lim}}
\renewcommand\limsup{\overline{\lim}}

\newcommand\ind{\mathds{1}}

\newcommand\eps{\varepsilon}

\newcommand{\vertiii}[1]{{\left\vert\kern-0.25ex\left\vert\kern-0.25ex\left\vert #1 
    \right\vert\kern-0.25ex\right\vert\kern-0.25ex\right\vert}}

\renewcommand{\leq}{\leqslant}

\renewcommand{\geq}{\geqslant}

\newtheorem{theorem}{Theorem}
\newtheorem{lemma}{Lemma}
\newtheorem{assumption}{Assumption}
\newtheorem{prop}{Proposition}

\newtheorem{remark}{Remark}
\newtheorem{definition}{Definition}

\iffalse    
   \excludecomment{corrige}
\else
  
\fi

\begin{document}

\title{A subexponential version of Cramér's theorem}
\author{Grégoire Ferré\\
\small Capital Fund Management, 23-25 rue de l'Université, 75007 Paris
}

\date{\today}

\maketitle

\abstract{
  We consider the large deviations associated with the empirical mean of independent
  and identically distributed
  random variables under a subexponential moment condition. We show that non-trivial
  deviations are observable at a subexponential scale in the number of variables, and
  we provide the associated rate function, which is non-convex and is not derived from a
  Legendre--Fenchel transform. The proof adapts the one of Cramer's
  theorem to the case where the fluctuation is generated by a single variable.
  In particular, we develop a new tilting strategy for the lower bound, which leads
  us to introduce a condition on the second derivative of the moment generating
  function. Our results are illustrated by a couple of simple examples.
}

\section{Introduction}

In most cases, the  empirical mean of independent and identically distributed random
variables converges to the expectation of the variable, according to the law of large numbers.
The central limit theorem (CLT) describes small fluctuations around the mean,
 which are gaussian and scale in the square root of the number of observations. 
A spectacular feature of the CLT is its universality: all variables with the same variance
have the same gaussian small fluctuations. Moreover, convergence rates towards the CLT are available, 
for instance through Berry-Esseen type bounds~\cite{nourdin2012normal}.

It is often interesting to control fluctuations far away from the
CLT regime, for both theoretical and practical reasons. This is
the concern of \emph{large deviations theory}, which provides such asymptotic
control typically at exponential scale~\cite{deuschel2001large,dembo2010large, touchette2009large}.
In a standard situation where the random variables have some finite exponential moment,
probabilities of fluctuations are indeed exponentially small with the number of observations,
the rate of smallness being controlled by a function called \emph{rate function}, which
is in general quadratic around the mean. As a result, a large deviations principle generalizes the
strong law of large numbers (by Borel-Cantelli) and the CLT (by expanding the rate function
around the mean).
Contrarily to the CLT, the rate function is not universal and depends a priori on the entire
distribution of the variable.

However, the exponential fluctuation scaling does not
always hold true with a non-trivial rate function. Actually,
the empirical mean can be controlled at an exponential scale if and only if
the variable has some exponential moment~\cite{wang2010sanov}.
When the variable does not have any exponential moment, the situation is much
more complicated. Fluctuation theorems for subexponential variables
have been investigated by Nagaev~\cite{nagaev1969integral,nagaev1969integralii}
and Borovkov~\cite{borovkov2000large} before being
recently revisited~\cite{gantert2014large,brosset2020probabilistic}.
Although these works provide useful subexponential estimates,
they are not precise enough to provide full large deviations principles with amenable rate
function as we could expect from the modern theory~\cite{dembo2010large}, in particular
concerning the lower bound.

In this paper, we prove a full large deviations principle for a class of subexponential
variables. Contrarily to previous works on the topic (see~\cite{gantert2014large,brosset2020probabilistic}
and references therein), we do not assume any form for
the cumulative distribution of the random variable, but rather work with a scaled
version of the cumulant function that encompasses subexponentialness.
This allows in particular to consider cases where the distribution at hand is not known.
For proving
the upper bound, we rely on the very insightful work~\cite{gantert2014large} that
we adapt to our moment assumption. For the lower bound, we develop a new tilting strategy by
using a subexponential transform on one variable. Since only one variable is tilted,
we cannot use standard concentration estimates, and rather control the deviations of the tilted
variable through an assumption on
the second derivative  of the moment generating function, which strengthens the Gartner--Ellis
condition. Although the upper bound sheds some light on the large deviations mechanism,
the proof of the lower bound is the most instructive and original part of the paper.
As a side product,  it also provides the optimal sampling strategy
for reducing variance of a large deviations estimator.

The work is organized as follows. Section~\ref{sec:main} presents our assumptions
and the associated large deviations result.
Some examples of application are proposed in Section~\ref{sec:examples}, while
the proofs of lower and upper bounds are postponed to Section~\ref{sec:proof}.
Some perspectives are finally  discussed in Section~\ref{sec:discussion}.

\vspace{0.5cm}


\section{Large deviations at subexponential scale}
\label{sec:main}

We consider i.i.d. samples~$(X_i)_{i\in\N^*}$ of a random variable~$X$
with law~$\mu$ on~$\R$, satisfying the following assumption.

\begin{assumption}
  \label{as:basics}
  The random variable~$X$ is symmetric and has finite polynomial moments of any order.
\end{assumption}
An immediate consequence of this assumption is that~$\E[X]=0$.
These conditions are not restrictive for the problem we are considering,
but simplify the presentation of the
results. We associate the samples~$(X_i)_{i\in\N^*}$ the empirical mean
\begin{equation}
  \label{eq:Sn}
S_n = \frac{1}{n} \sum_{i=1}^n X_i.
\end{equation}

Our goal is to derive a large deviations principle for~$S_n$
under a subexponential moment condition on~$X$. For this, we introduce
the following scaling functions. 
\begin{definition}
  \label{def:scale}
  For $\alpha\in(0,1)$, we define a scaling function by
  \[
  \forall\, x\in\R,\quad \phi_\alpha(x) = \mathrm{sign}(x)|x|^\alpha.
  \]
\end{definition}
The scaling function~$\phi_\alpha$ is a natural tool to introduce
a subexponential moment condition through a scaled free
energy\footnote{We prefer the naming free energy to cumulant generating
  function because \emph{scaled cumulant generating function} refers
  to another concept related to the Gartner--Ellis
  theorem~\cite{dembo2010large}.}, which we define, for any $\alpha\in(0,1)$, by
\begin{equation}
  \label{eq:scaledFE}
  \forall\,\eta\in\R,\quad
  \lambda_\alpha(\eta) = \log \E\left[ \e^{\eta \phi_\alpha(X)}\right].
\end{equation}
Just like a standard free energy,~$\lambda_\alpha$ is a convex function from~$\R$
into~$(-\infty,+\infty]$ with domain
\begin{equation}
  \label{eq:domaindef}
D_{\lambda_\alpha} = \{ \eta\in\R,\, \lambda_\alpha(\eta)<+\infty\}.
\end{equation}
Since~$\lambda_\alpha$ is convex, its domain is convex, hence it is
a segment. Since~$X$ is symmetric so is~$\lambda_\alpha$. We thus
have~$D_{\lambda_\alpha} = (-\xi,\xi)$ or $D_{\lambda_\alpha} = [-\xi,\xi]$
for some~$\xi\in [0,+\infty]$. Moreover, we know that~$\lambda_\alpha$ is
(infinitely) differentiable on the interior of its domain by standard dominated
convergence results~\cite[Lemma~2.2.5]{dembo2010large}.
We now propose a generalization of the
essential smoothness condition of the Gartner--Ellis theorem
(see~\cite[Definition~2.3.5]{dembo2010large}).
\begin{assumption}[Second order essential smoothness]
  \label{as:essmooth}
  There exists~$\alpha\in(0,1)$ such that the function
  $\lambda_\alpha:\R\to (-\infty,+\infty]$ satisfies:
    \begin{itemize}
    \item Non-trivial bounded domain:
      $D_{\lambda_\alpha} = (-\xi,\xi)$ for some~$\xi\in (0,+\infty)$.
    \item Steepness: for any sequence~$(\eta_n)_{n\in\N}$ converging to~$\pm\xi$,
      it holds
      \begin{equation}
        \label{eq:lambdasteep}
        |\lambda_\alpha'(\eta_n)|\xrightarrow[n\to+\infty]{} +\infty.
      \end{equation}
    \item Bounded relative variance: define
      \begin{equation}
        \label{eq:relatvar}
        \mathcal{V}(\eta)= \frac{\lambda_\alpha''(\eta) }{\lambda_\alpha'(\eta)^2}.
      \end{equation}
      There exist~$\xi_0\in(0,\xi)$ and~$\omega\in (0,+\infty)$ such that~$\lambda_\alpha''$
        (resp.~$\mathcal{V}$)
        is non-decreasing (resp. non-increasing) on~$[\xi_0,\xi)$ and
          \begin{equation}
            \label{eq:Vsteep}
            \forall\, \eta\in(\xi_0,\xi), \quad
          \mathcal{V}(\eta)\leq \omega.
          \end{equation}
    \end{itemize}

\end{assumption}

We are now in position to state our main theorem. The definition of a large deviations
principle is recalled (with some technical details) in Appendix~\ref{sec:appendix}.
\begin{theorem}
  \label{th:main}
  Let Assumptions~\ref{as:basics} and~\ref{as:essmooth} hold. Then the empirical mean~$S_n$
  defined in~\eqref{eq:Sn} satisfies a large deviations principle at speed~$n^\alpha$
  with rate function~$I_\alpha$ defined by:
  \begin{equation}
    \label{eq:Ialpha}
    \forall\,x\in\R, \quad I_\alpha(x) = \xi |x|^\alpha.
  \end{equation}
\end{theorem}
The proof of Theorem~\ref{th:main} is postponed to Section~\ref{sec:proof}. We propose
some remarks on this result before presenting in Section~\ref{sec:examples}
a couple of situations where it applies.

\begin{remark}
  \begin{itemize}
\item Although, for i.i.d. variables,~$\lambda_\alpha$ is differentiable on the interior
  of its domain (so we do not need to make a smoothness assumption on~$\lambda_\alpha$),
  there exist non-steep such free energies~\cite[exercize 2.3.17]{dembo2010large}.
\item One may be confused by the reference to Cramér's theorem, whereas we use a generalization
  of the Gartner--Ellis steepness condition. We refer here to the range of applications of Theorem~\ref{th:main},
  which concerns independent variables, rather than the assumption. Moreover, the steepness
  condition is not used for deriving the lower bound from the upper bound through convex
  analysis arguments, but to perform an arbitrarily large exponential tilting on one variable, which
  is quite different.
\item If~$X$ is a variable with density~$p$ whose tail scales like~$\e^{-\xi_\alpha |x|^\alpha}$
  at infinity, then the scaled variable $Y=\phi_\alpha(X)$ mostly scales
  like~$\e^{-\xi_\alpha |y|}$ at infinity, up to slowly varying functions at exponential
  scale, like in~\cite{gantert2014large}. One can thus understand
  Assumption~\ref{as:essmooth} as a way to find
  the right scaling to bring the subexponential tail back to an exponential one (or in other words
  to find the correct subexponential decay scale). The speed of the LDP is given by the exponent~$\alpha$
  while the rate function is fully determined by the tail factor~$\xi_\alpha$.
\item One can consider the scaled random variable $Y=\phi_\alpha(X)$.
  Since this variable has an exponential moment, its empirical mean satisfies a LDP with a good rate
  function defined by the Fenchel transform
  \[
  J_\alpha(x) = \sup_\eta\, \{\eta x - \lambda_\alpha(\eta)\}.
  \]
  We can check by convex analysis that actually
  \[
  \xi = \lim_{n\to +\infty}\ \frac{J_\alpha(n)}{n},
  \]
  and thus, $\forall\, x \in\R$,
  \[
  I_\alpha(x) = \lim_{n\to +\infty}\ \frac{J_\alpha(n^\alpha|x|^\alpha)}{n^\alpha}.
  \]
  One could actually expect the subexponential rate function~$I_\alpha$ to be equal
  to~$J_\alpha(|\,\cdot\,|^\alpha)$ because~$Y$ is equal to~$X$ transported
  by the mapping~$\phi_\alpha$. However, the fact that the fluctuation is most typically
  generated by one variable makes only the tail of~$J_\alpha$ asymptotically visible.
  This is another interpretation of the coefficient~$\xi$ in~\eqref{eq:Ialpha}.
\item In general,~\eqref{eq:Vsteep} is not a necessary condition for Theorem~\ref{th:main}
  to hold. From the proof of the lower bound, a closer-to-optimal condition
  might be:
  \[
    \forall\,x>0,\quad
    \mathcal{V}(\eta_n^x)\underset{n\to \infty}{=}\mathrm{o}(n^\alpha),\quad
    \mathrm{with} \quad \eta_n^x = (\lambda_\alpha')^{-1}\big((nx)^\alpha\big).
  \]
  However it does not seem necessary
  to reach such a precision in the cases we are interested in.
\item The monotonicity assumptions on~$\lambda_\alpha''$ and~$\mathcal{V}$ are set for
  pure convenience in order to simplify the last step of the proof of the lower bound. The important
  part of the assumption is the boundedness of~$\mathcal{V}$.
\item We assume that~$\xi\in(0,+\infty)$ for simplicity but Theorem~\ref{th:main}
  also holds when $\xi =0$ or $\xi = +\infty$. In such a situation
  the LDP at scale~$n^\alpha$ is trivial. We therefore avoid distinguishing cases and focus on the non-trivial situation.
  In particular, Theorem~\ref{th:main} is consistent with the standard Cramér's
  theorem at exponential scale~\cite[Theorem~2.2.3]{dembo2010large},
  since only one LDP scaling provides a non-trivial result.
  \end{itemize}
\end{remark}

\section{Simple applications}
\label{sec:examples}
Before diving into the proof, we propose a couple of illustrative applications that 
actually motivated this study. The baseline is to consider a simple random variable and
to raise it to some power~$p>0$. In general, for~$p$ small enough, a standard large deviations
principle holds, while for~$p$ large we can use Theorem~\ref{th:main}, which uncovers
a phase transition. The physical idea behing these examples is to observe of function
of interest over a simple system.

Let us start with the exponential case.
Let~$Y$ be a two-sided exponential random variable with distribution on~$\R$
given by
\begin{equation}
  \label{eq:twosidedexpo}
\nu(dy) = \frac{\e^{-|y|}}{2}dy,
\end{equation}
and consider $X = \phi_p(Y)$ for~$p\geq 0$. In this case, the large deviations of~$X$
for $p\in(0,1]$ are covered by Cramer's theorem. This result however fails to provide a
useful information for $p>1$, since then~$X$
does not have any exponential moment any more. Theorem~\ref{th:main} allows to get the full
picture on this situation.

\begin{prop}[Powers of exponential variables]
  \label{prop:expo}
  For~$p\in(0,1]$, the empirical mean of~$X$ satisfies a large deviations principle at speed~$n$
  with rate function given by
  \[
  J(x) = \sup_\eta\, \{ \eta x - \lambda(\eta)\} \quad \mathrm{where}\quad
  \lambda(\eta) = \log\E\left[\e^{\eta X}\right],
  \]
  while for $p>1$ a large deviations principle holds at speed~$n^{1/p}$ with rate function
  \[
  I_{1/p}(x) = |x|^{1/p}.
  \]
\end{prop}

A similar result can easily be obtained for powers of symmetrized Gamma random variables.
We consider instead $Z=\phi_p(G)$ where~$G$ is a standard Gaussian
random variable for $p>0$.

\begin{prop}[Powers of Gaussian variables]
  \label{prop:gaussian}
  For any~$p\in(0,2]$, the empirical mean of~$Z$ satisfies a large deviations principle at speed~$n$
  with rate function given by
  \[
    J(x) = \sup_\eta\, \{ \eta x - \lambda(\eta)\} \quad \mathrm{where}\quad
  \lambda(\eta) = \log\E\left[\e^{\eta Z}\right],
  \]
  while for $p>2$, a large deviations principle holds at speed~$n^{2/p}$ with rate function
  \[
  I_{2/p}(x) = \frac{|x|^{2/p}}{2}.
  \]
\end{prop}

\begin{proof}[Proof of Propositions~\ref{prop:expo} and~\ref{prop:gaussian}]
  When~$X=\phi_p(Y)$ with~$Y$ defined by~\eqref{eq:twosidedexpo}, the variable~$X$ has
  exponential moments for any $p\leq 1$, so we focus on~$p>1$. In this case,
  setting~$\alpha = 1/p<1$, a simple computation shows that
  \[
\lambda_\alpha(\eta) = \log \E[\e^{\eta \phi_\alpha(X)}]= \log \E[\e^{\eta Y}] = - \log( 1 - \eta^2).
\]
We can then check the criteria of Assumption~\ref{as:essmooth} by first noting
that $D_{\lambda_\alpha} = (-1,1)$ and
\[
\forall\,\eta \in D_{\lambda_\alpha},\quad
\lambda_\alpha'(\eta) = \frac{2\eta}{1 - \eta^2},\quad
\lambda_\alpha''(\eta) = 2\frac{(1+\eta^2)}{(1-\eta^2)^2}.
\]
As a results,~\eqref{eq:lambdasteep} holds and
\[
\mathcal{V}(\eta) = \frac{1+\eta^2}{2\eta^2}\xrightarrow[\eta\to 1]{} 1,
\]
so that~\eqref{eq:Vsteep} is satisfied with $\omega=1+\eps$ for any $\eps>0$.
Since it is clear that~$\lambda_\alpha''$ is
increasing (and~$\mathcal{V}$ decreasing), all the conditions of Assumption~\ref{as:essmooth}
are satisfied and Theorem~\ref{th:main} applies.

Now, when $Z=\phi_p(G)$, the variable~$Z$ has exponential moments for $p\leq 2$, we thus consider the case
$p >2$. In this case we set $\alpha = 2/p<1$ and we can show that
\[
\lambda_\alpha(\eta) = \log \E[\e^{\eta \phi_\alpha(Z)}] =  \log \E[\e^{\eta \mathrm{sign}(G) G^2}] =
\log( \frac{1}{\sqrt{1 + 2\eta}} + \frac{1}{\sqrt{1 - 2\eta}}  ) - \log\left(
\frac{2}{\sqrt{2}}\right).
\]
We thus have $D_{\lambda_\alpha}= (-1/2,1/2)$. Introducing
\[
f(\eta)=(1 + 2\eta)^{-1/2} + (1 - 2\eta)^{-1/2},
\]
we compute
\[
f'(\eta) = -(1 + 2\eta)^{-3/2} + (1 - 2\eta)^{-3/2},\quad f''(\eta) =
3(1 + 2\eta)^{-5/2} + 3(1 - 2\eta)^{-5/2}.
\]
As a result, we get
\[
\lambda_\alpha'(\eta) = \frac{f'(\eta)}{f(\eta)}\underset{\eta\to 1/2}{\sim}
(1 - 2\eta)^{-1}\xrightarrow[\eta\to 1/2]{} +\infty,
\]
and a symmetric conclusion holds for $\eta\to - 1/2$, so~\eqref{eq:lambdasteep} holds.
In a similar fashion,
\[
\lambda_\alpha''(\eta) = \frac{f(\eta)f''(\eta) - f'(\eta)^2  }{f(\eta)^2},
\]
so
\[
\mathcal{V}(\eta) = \frac{f(\eta)f''(\eta)  }{f'(\eta)^2} - 1
\underset{\eta\to 1/2}{\sim} 3 \frac{(1 - 2\eta)^{-5/2}(1 - 2\eta)^{-1/2}}{(1 - 2\eta)^{-6/2}}
-1\xrightarrow[\eta\to 1/2]{} 2.
\]
This entails that~\eqref{eq:Vsteep} again holds with $\omega=2+\eps$ for any $\eps>0$.
The monotonicity conditions on~$\lambda_\alpha''$ and~$\mathcal{V}$ are also easily checked so
all the conditions of Assumption~\ref{as:essmooth} hold and Proposition~\ref{prop:gaussian}
is a consequence of Theorem~\ref{th:main}.
\end{proof}

\section{Proof of Theorem~\ref{th:main}}
\label{sec:proof}

We now present the proof of Theorem~\ref{th:main}, which starts with the lower
bound.

\subsection{Proof of the lower bound}
It is standard for the following condition to hold
 to prove the lower bound: 
\[
\forall\, x \in\R, \quad \lim_{\delta\to 0}\ \underset{n\to +\infty}{\liminf}\
\frac{1}{n^\alpha}\log
\proba\big( S_n \in B(x,\delta) \big) \geq - I_\alpha(x),
\]
where $B(x,\delta)=(x-\delta, x+\delta)$.
By symmetry we can restrict ourselves to~$x>0$.
Let then $x,\delta,\eps >0$ be arbitrary and compute
\[
\begin{aligned}
\proba\big( S_n \in B(x,\delta) \big)
& = \proba\left(x-\delta\leq \frac{X_1}{n} + \frac 1n \sum_{i=2}^n X_i \leq x+\delta\right)
\\ & \geq  \proba\left(x-\delta -\eps\leq \frac{X_1}{n}  \leq x+\delta + \eps ,\
-\eps \leq \frac 1n \sum_{i=2}^n X_i\leq \eps \right)
\\ & = \proba\big( nx_- \leq X_1  \leq nx_+\big)
\,\proba\left( -\eps \leq \frac 1n \sum_{i=2}^n X_i\leq \eps \right).
\end{aligned}
\]
We introduced in the last line the notation $x_\pm = x \pm(\delta+\eps)$, which
we shall use again below. By the law of large numbers and because $\E[X]=0$ while~$X$
has a finite second moment,
the second probability in the last line converges to one, so
we can focus on the first probability: the one of fluctuation of one variable.

We revisit the exponential (Esscher) transform used in the
theorems of Cramér and Gartner--Ellis by modifying several of its
main features.
The plan of the proof below is as follows:
\begin{itemize}
\item Perform a subexponential transform on one variable;
\item Find the optimal tilting parameter and the rate function;
\item Derive concentration estimates to ensure that the tilted
  variable asymptotically has the correct mean with controlled variance
  (this last part is
  itself split in two steps).
\end{itemize}

\subsubsection*{Single variable Esscher transform at subexponential scale}
We start by tilting the variable~$X_1$ at a subexponential scale. Let~$\eta\in (0,\xi)$
be arbitrary and write
\[
\begin{aligned}
\proba\big( n x_- \leq X_1  \leq nx_+\big)
 & = \int_{nx_-\leq y\leq nx_+}\e^{-\eta \phi_\alpha(y) +
  \eta \phi_\alpha(y)}\mu(dy)
\\
& \geq  \e^{-\eta n^\alpha x_+^\alpha}
   \int_{n x_-\leq y\leq nx_+}\e^{
     \eta \phi_\alpha(y)}\mu(dy)
   \\ & = \e^{-\eta n^\alpha x_+^\alpha + \lambda_\alpha(\eta)}
   \proba\big(n x_-\leq\tilde
   X_\eta \leq n x_+ \big),
\end{aligned}
\]
where we used the scaled free energy~\eqref{eq:scaledFE} to introduce the
tilted random variable~$\tilde X_\eta$ with law
\begin{equation}
  \label{eq:mutilde}
\tilde \mu_\eta(dy) =
\e^{\eta\phi_\alpha(y) - \lambda_\alpha(\eta)}\mu(dy).
\end{equation}
We recall that~$\mu$ is the law of~$X$.
This resembles the standard tilting technique but on one variable, and with the
terms inside the exponential scaled by~$\phi_\alpha$.

By also applying the increasing function~$\phi_\alpha$ in the probability,
we thus reach
\begin{equation}
  \label{eq:steplowerbound}
\proba\big( n x_- \leq X_1  \leq nx_+\big)
\geq \e^{-\eta n^\alpha x_+^\alpha
  + \lambda_\alpha(\eta)}\,\proba\big(n^\alpha \phi_\alpha(x_-)\leq
   \phi_\alpha(\tilde X_\eta) \leq n^\alpha \phi_\alpha(x_+)\big).
\end{equation}
In the following, we use that, for~$\delta,\eps$ small enough, it holds
$x_->0$ and so $\phi_\alpha(x_-) = x_-^\alpha$.
We now have to choose a sequence~$\eta_n$
such that~$\eta_n\to \xi$ and the last probability has an appropriate lower bound.
In other words, we have to find the parameter~$\eta$ that makes the fluctuation~$x$
most likely for~$\tilde X_\eta$ at minimal entropic cost.

\subsubsection*{Optimal tilting}
It is natural in such a proof to look for a critical point of
$\e^{-\eta n^\alpha x^\alpha + \lambda_\alpha(\eta)}$. Heuristically,
such a critical point depends on~$n$ and should satisfy
\[
\eta_n \in \underset{\eta\in(-\xi,\xi)}{\mathrm{argmax}}\,\big\{ \eta (n x)^\alpha - \lambda_\alpha(\eta)\big\},
\]
which is a scaled Legendre--Fenchel transform.
Actually, by Assumption~\ref{as:essmooth}, the function~$\lambda_\alpha$ is differentiable
and its derivative is one-to-one from~$(-\xi,\xi)$ into~$\R$. Therefore, the unique
critical point within~$(-\xi,\xi)$ of the function inside brackets above is well-defined by:
\begin{equation}
  \label{eq:etan}
\eta_n = (\lambda_\alpha')^{-1}\big( (n x)^\alpha\big).
\end{equation}
Since~$nx\to+\infty$, we see that $\eta_n\to \xi$.
In order to make~\eqref{eq:etan} more explicit, we use a standard
dominated convergence theorem together with~\eqref{eq:mutilde} to obtain that
\[
\forall\, \eta\in(-\xi,\xi),\quad \lambda_\alpha'(\eta)=
 \frac{\E\left[\phi_\alpha(X)\, \e^{\eta \phi_\alpha(X)} \right]}
  {\E\left[\e^{\eta \phi_\alpha(X)} \right]}=\E\big[\phi_\alpha(\tilde X_{\eta})\big].
\]
 The choice~\eqref{eq:etan} thus ensures that
\[
\E\big[\phi_\alpha(\tilde X_{\eta_n})\big]=\lambda_\alpha'(\eta_n) = (n x)^\alpha.
\]
It is an intriguing feature that the tilting does not make the average
of the tilted variable~$\tilde X_{\eta_n}$ to be equal to~$nx$ but rather works with~$\phi_\alpha(\tilde X_{\eta_n})$,
which is why we write the last probability in~\eqref{eq:steplowerbound} in this way.
In this situation, since~$\eta_n\to \xi$ and $\lambda_\alpha \geq 0$ by symmetry, the lower
bound~\eqref{eq:steplowerbound} actually becomes
\begin{equation}
  \label{eq:nearlowerbound}
\underset{n\to+\infty}{\liminf}\ \frac{1}{n^\alpha}\log
\proba\big( n x_- \leq X_1  \leq nx_+\big)
\geq -\xi  x_+^\alpha
  + \underset{n\to+\infty}{\liminf}\ \frac{1}{n^\alpha}\log \proba\big(n^\alpha \phi_\alpha(x_-)\leq
   \phi_\alpha(\tilde X_{\eta_n}) \leq n^\alpha \phi_\alpha(x_+)\big).
\end{equation}
Thus the lower bound holds (by letting $\eps,\delta\to 0$) provided we can
control the last probability in the above inequality.

However, although the average of~$\phi_\alpha(\tilde X_{\eta_n})$ is correct, we cannot obtain a straightforward
control of $\proba\big(n^\alpha \phi_\alpha(x_-)\leq
\phi_\alpha(\tilde X_{\eta_n}) \leq n^\alpha \phi_\alpha(x_+)\big)$
 because there is only one variable, so using
 a law of large numbers as is usual for Cramér-like results is not an option.
 Applying the upper bound to control this term like in the Gartner--Ellis theorem
 also looks difficult (again because we do not manipulate an average).
 This is why we control the standard
 deviation of~$\phi_\alpha(\tilde X_{\eta_n})$ with explicit bounds through
 the last point of Assumption~\ref{as:essmooth}.

\subsubsection*{Concentration through second derivative}

For the sake of simplicity we introduce $\gamma>0$ defined by
\[
\gamma^\alpha = \min (x_+^\alpha - x^\alpha ,   x^\alpha - x_-^\alpha).
\]
This number  goes to zero as~$\eps+\delta$ goes to zero (we don't write
explicitly the dependency to avoid overloading notation).
We can thus write
\[
\begin{aligned}
\proba\big((nx_-)^\alpha\leq
\phi_\alpha(\tilde X_{\eta_n}) \leq (nx_+)^\alpha\big)
&\geq \proba\big( - ( n\gamma)^\alpha\leq
   \phi_\alpha(\tilde X_{\eta_n})- (nx)^\alpha \leq  (n\gamma)^\alpha  \big)
\\ & = \proba\left( \big|
   \phi_\alpha(\tilde X_{\eta_n})-\E[\phi(\tilde X_{\eta_n})]\big| \leq  ( n\gamma)^\alpha \right)
\\ & = 1 - \proba\left( \big|
   \phi_\alpha(\tilde X_{\eta_n})-\E[\phi(\tilde X_{\eta_n})]\big| >  ( n\gamma)^\alpha \right).
\end{aligned}
\]
Our goal is now to control the last probability\footnote{Markov's inequality
  implies that
  \[
\proba\left( \big|
\phi_\alpha(\tilde X_{\eta_n})-\E[\phi(\tilde X_{\eta_n})]\big| >  ( n\gamma)^\alpha \right)
\leq \frac{\E\left[\left(\phi_\alpha(\tilde X_{\eta_n})-\E[\phi(\tilde X_{\eta_n})] \right)^2\right]}
{( n\gamma)^{2\alpha}} = \left(\frac{x}{\gamma}\right)^{2\alpha}\frac{\lambda_\alpha''(\eta_n)}{\lambda_\alpha'(\eta_n)^2}
=\left(\frac{x}{\gamma}\right)^{2\alpha}\mathcal{V}(\eta_n).
\]
Therefore, if~$\mathcal{V}(\eta_n)\to0$ for any sequence $\eta_n\to\xi$, we reach the
desired result (recall that~$x$ is fixed and~$\gamma$ is small). However, this
assumption does not seem to be applicable in practical
cases, which is why we have to consider the weaker condition~\eqref{eq:Vsteep} and compute
more precise estimates at exponential scale.
}.
For this we rely on a (easily proved) symmetrized Tchebychev inequality: for
any random variable~$Z$ and $a,k>0$ it holds
\begin{equation}
\label{eq:maxproba}
  \proba ( |Z - \E[Z]| >a ) \leq \max\left( \E\left[\e^{k ( Z - \E[Z] - a)}\right],
\  \E\left[\e^{k ( -Z + \E[Z] + a)}\right] 
    \right).
\end{equation}
By symmetry we can consider one case only. Let us focus on the first one and
choose an arbitrary $k_n\in(0,\xi - \eta_n)$. Taking $Z= \phi_\alpha(\tilde X_{\eta_n})$ and
recalling that here $\E[Z] = (nx)^\alpha$ and $a=(n\gamma)^\alpha$, we have
\[
\E\left[\e^{k_n ( Z - \E[Z] - a)}\right] = \e^{-k_n n^\alpha( x^\alpha + \gamma^\alpha)}
\frac{\E\left[\e^{k_n \phi_\alpha( X)}\e^{\eta_n
      \phi_\alpha(X)}\right]}
       {\E\left[\e^{\eta_n \phi_\alpha( X)}\right]}
       =
       \e^{-k_n n^\alpha \tilde x_+^\alpha + \lambda_\alpha(k_n + \eta_n)
       - \lambda_\alpha( \eta_n)},
       \]
where we  introduced $\tilde x_+ = (x^\alpha + \gamma^\alpha)^{\frac1\alpha}>x$.
Since~$k_n\in(0,\xi - \eta_n)$, we set $y_n=k_n+\eta_n \in(\eta_n,\xi)$, which leads to
\[
-k_n (n\tilde x_+)^\alpha + \lambda_\alpha(k_n + \eta_n)
- \lambda_\alpha( \eta_n) = \eta_n  (n\tilde x_+)^\alpha - \lambda_\alpha(\eta_n)
- \big(y_n  (n\tilde x_+)^\alpha -\lambda_\alpha(y_n)\big). 
\]
When optimizing over $y_n \in(\eta_n,\xi)$,
the infimum of the above quantity is attained inside~$(\eta_n,\xi)$ (easily proved)
at the value
\[
y_n = \tilde \eta_n = (\lambda_\alpha')^{-1}\big( (n\tilde x_+)^\alpha\big).
\]
By steepness of~$\lambda_\alpha$ the above quantity is well-defined, and by
convexity (hence monotonicity of $\lambda_\alpha'$), the inequality $\tilde x_+> x$
implies that $\tilde \eta_n \geq \eta_n$.

Since the second case in~\eqref{eq:maxproba} is symmetric, we obtain
\[
\log \proba ( |Z - \E[Z]| >a ) \leq \lambda_\alpha(\tilde \eta_n)
+ ( \eta_n -\tilde \eta_n)\lambda_\alpha'(\tilde \eta_n) - \lambda_\alpha(\eta_n).
\]
We see that the above quantity looks like a Taylor expansion  of~$\lambda_\alpha$
at first order. We thus expand the cumulant function (backward) as follows:
there exists~$\bar \eta \in[\eta_n,\tilde \eta_n]$ such that
\[
\lambda_\alpha(\eta_n)= \lambda_\alpha(\tilde \eta_n + (\eta_n - \tilde \eta_n))
= \lambda_\alpha(\tilde \eta_n) + \lambda_\alpha'(\tilde \eta_n) (\eta_n - \tilde \eta_n)
+\frac12 \lambda_\alpha''( \bar \eta) (\eta_n - \tilde \eta_n)^2.
\]
As a result, since~$\lambda_\alpha''$ is increasing close enough to~$\xi$ and $\eta_n\leq \bar \eta$,
we have
\begin{equation}
  \label{eq:ineqlambdasec}
\log \proba ( |Z - \E[Z]| >a )\leq - \frac12 \lambda_\alpha''( \eta_n)
(\eta_n - \tilde \eta_n)^2.
\end{equation}
In order to control~\eqref{eq:ineqlambdasec}, we see at this stage a competition
between~$\lambda_\alpha''(\eta_n)$ that typically diverges to infinity
and $(\eta_n - \tilde \eta_n)^2$, which goes to zero. Let us derive the estimates on
this second term to reach the desired conclusion.

\subsubsection*{Convex analysis for variance control}

We first introduce the Legendre transform of~$\lambda_\alpha$:
\[
J_\alpha(x) = \sup_{\eta}\,\{ \eta x - \lambda_\alpha(\eta)\}.
\]
Standard convex analysis~\cite[Chapter~VI]{ellis2007entropy} shows that
\[
(\lambda_\alpha')^{-1} (\,\cdot\,)= J_\alpha'(\,\cdot\,).
\]
Therefore
\[
\tilde \eta_n - \eta_n = J_\alpha'\big( (n \tilde x_+)^\alpha\big) - J_\alpha'\big( (n x)^\alpha\big).
\]
We then perform another expansion but at order one: there is
$b_n \in [ (nx)^\alpha,  (n\tilde x_+)^\alpha]$ such that
\begin{equation}
  \label{eq:firstbb}
\tilde \eta_n - \eta_n = n^\alpha (\tilde x_+^\alpha - x^\alpha) J_\alpha''(b_n) \geq (n\gamma)^\alpha J_\alpha''(b_n) .
\end{equation}
We now use~\cite{crouzeix1977relationship}
to relate the second derivative of~$J_\alpha$ to the one of~$\lambda_\alpha$:
\begin{equation}
  \label{eq:secondbb}
J_\alpha''(b_n) = \frac{1}{\lambda_\alpha''\big( (\lambda_\alpha')^{-1}(b_n)\big)}
\geq \frac{1}{\lambda_\alpha''\big( (\lambda_\alpha')^{-1}\big( (n\tilde x_+)^\alpha\big)\big)}
= \frac{1}{\lambda_\alpha''( \tilde \eta_n)},
\end{equation}
where we used the monotonicity of~$\lambda_\alpha'$ and~$\lambda_\alpha''$
for~$n$ large enough to obtain the inequality above. 
Combining~\eqref{eq:firstbb} with~\eqref{eq:secondbb}, we can
turn~\eqref{eq:ineqlambdasec} into
\begin{equation}
  \label{eq:logratioint}
\log \proba ( |Z - \E[Z]| >a )\leq - \frac12 \lambda_\alpha''(\eta_n)\left(
\frac{(n\gamma)^\alpha}{\lambda_\alpha''(\tilde \eta_n)}
\right)^2.
\end{equation}
Introducing
\[
c_x =  \frac{ (x\gamma)^{2\alpha}}{2 ( \tilde x_+)^{4\alpha}}>0,
\]
then~\eqref{eq:logratioint} may be arranged as
(recalling that $(nx)^\alpha = \lambda_\alpha'(\eta_n)$
and similarly $(n\tilde x_+)^\alpha = \lambda_\alpha'(\tilde \eta_n)$):
\[
\begin{aligned}
\log \proba ( |Z - \E[Z]| >a )& \leq - c_x \frac{\lambda_\alpha''( \eta_n)}
     {(nx)^{2\alpha}}
     \left(
\frac{(n\tilde x_+)^{2\alpha}}{\lambda_\alpha''( \tilde \eta_n)}
\right)^2
 \\ & =- c_x \frac{\lambda_\alpha''(\eta_n)}
     {\lambda_\alpha'(\eta_n)^2}
     \left(
\frac{\lambda_\alpha'(\tilde \eta_n)^{2}}{\lambda_\alpha''( \tilde \eta_n)}
\right)^2
\\ & = - c_x \frac{\mathcal{V}(\eta_n)}{\mathcal{V}(\tilde \eta_n)^2}
  \leq - \frac{ c_x}{\mathcal{V}(\tilde \eta_n)}
\leq -\frac{c_x}{\omega},
\end{aligned}
\]
where, in
Assumption~\ref{as:essmooth}, we used monotonicity of~$\mathcal{V}$ to get
$-\mathcal{V}(\eta_n)\leq-\mathcal{V}(\tilde \eta_n)$
 for~$n$ large enough, as well as~\eqref{eq:Vsteep} for the last inequality.
 Note that when $\mathcal{V}(\tilde \eta_n)\to 0$, the probability of
 deviating from the mean converges to zero, but in general it is just smaller than one.
 We thus obtain by~\eqref{eq:maxproba} that there exists~$c_{x,\omega}>0$ (depending also
 on the fixed parameters~$\eps,\delta>0$) such that, for~$n$ large enough, it holds
\[
\proba\big(n^\alpha \phi_\alpha(x_-)\leq
\phi_\alpha(\tilde X_{\eta_n}) \leq n^\alpha \phi_\alpha(x_+)\big)\geq c_{x,\omega}.
\]
Plugging this estimate in~\eqref{eq:nearlowerbound}
allows to conclude the proof of the lower bound.

\subsection{Upper bound}
\label{sec:upper}
We now turn to the upper bound, for which it is sufficient~\cite[Theorem~2.2.3]{dembo2010large}
to study the
probability~$\proba\left( S_n \geq x\right)$ for $x>0$. We thus fix $x >0$, and first write
\begin{equation}
  \label{eq:upperseparation}
\proba\left( S_n \geq x\right) \leq
\proba\left( \max_{1\leq i\leq n} X_i \geq nx\right)
+ \proba\left(\max_{1\leq i\leq n} X_i < nx,\ \frac 1n \sum_{i=1}^n X_i \geq x\right)
= A_n^1 + A_n^2.
\end{equation}
We recall that, from~\cite[Lemma~1.2.15]{dembo2010large}, we have
\begin{equation}
  \label{eq:maxloglim}
\underset{n\to +\infty}{\limsup}\  \frac{1}{n^\alpha} \log (A_{n}^1 +A_{n}^2)
= \max \left( \underset{n\to +\infty}{\limsup}\  \frac{1}{n^\alpha}\log A_{n}^1,
\, \underset{n\to +\infty}{\limsup}\ \frac{1}{n^\alpha}\log A_{n}^2
\right).
\end{equation}
We can therefore study the sequences~$A_n^1$ and~$A_n^2$ separately and
take the maximum of the two when going at logarithmic scale. We closely follow
the path of~\cite{gantert2014large} by generalizing some elements along
Assumption~\ref{as:essmooth}.

\subsubsection*{Large deviations for the heavy tail term~$A_n^1$.}
For the first term we use the union's bound together with
Tchebychev's inequality at subexponential scale to obtain, for any $\eta\in(0,\xi)$:
\[
\proba\left( \max_{1\leq i\leq n} X_i \geq nx\right)
\leq n \proba\left( X_1 \geq nx\right)
\leq n \e^{- \eta (n x)^\alpha}\,\e^{ \lambda_\alpha(\eta)}.
\]
Therefore,
\[
 \frac {1}{n^\alpha}\log\proba\left( \max_{1\leq i\leq n} X_i \geq nx\right)
\leq -\eta |x|^\alpha+ \frac{\lambda_\alpha(\eta)}{n^\alpha} + \frac{\log(n)}{n^\alpha}.
\]
Since $\eta\in (0,\xi)$ is fixed, Assumption~\ref{as:essmooth} implies that
$\lambda_\alpha(\eta)<+\infty$, so
\[
\underset{n\to +\infty}{\limsup}\
 \frac {1}{n^\alpha}\log\proba\left( \max_{1\leq i\leq n} X_i \geq nx\right)
\leq -\eta |x|^\alpha.
\]
We now can pass to the limit $\eta\to\xi$ to obtain that
\[
\underset{n\to +\infty}{\limsup}\
\frac {1}{n^\alpha}\log A_n^1
\leq -\xi |x|^\alpha.
\]

\subsubsection*{Controlling the light tail term~$A_n^2$.}
Let us turn to the second term in~\eqref{eq:upperseparation}. The idea now is to
use Tchebychev's inequality at exponential scale but with a
parameter~$\beta_n>0$ depending on~$n$:
\[
A_n^2 \leq \e^{-\beta_n x}
\E\left[ \ind_{\left\{\overset{n}{\underset{i=1}{\max}}\,
    X_i < nx \right\}}\e^{\frac{\beta_n}{n}\sum_{i=1}^n X_i }
  \right]
\leq \e^{-\beta_n x}\prod_ {i=1}^n 
\E\left[ \ind_{\{ X_i < nx \}}\e^{\frac{\beta_n}{n} X_i }
  \right].
\]
It is natural to choose $\beta_n = n^\alpha \theta$ for some $\theta> 0$,
since we then obtain
\begin{equation}
  \label{eq:tchebn}
\frac{1}{n^\alpha} \log A_n^2 \leq - \theta x + n^{1-\alpha}
 \log \E\left[\ind_{\{ X_1 < nx \}} \e^{\theta n^{\alpha-1}X_1 
}  \right].
\end{equation}
It is tempting to set $\theta = \xi x^{\alpha - 1}$, however this will be a limit
case. We actually need a precise control of the remaining term, in which the bound on~$X_i<nx$
going to infinity comes in competition with the factor~$\theta n^{\alpha -1}$ inside
the exponential, which goes to zero. Following~\cite{gantert2014large}, we then prove
the following lemma.

\begin{lemma}
  \label{lem:boundremainder}
  For any $\theta < \xi x^{\alpha -1}$ it holds
  \[
  \underset{n\to +\infty}{\limsup}\ n^{1-\alpha}
 \log \E\left[\ind_{\{ X_1 < nx \}} \e^{\theta n^{\alpha-1}X_1}  \right] \leq 0.
\]  
\end{lemma}

If we prove Lemma~\ref{lem:boundremainder} then we can take the limit
$\theta\to \xi x^{\alpha-1}$ in~\eqref{eq:tchebn} and~\eqref{eq:maxloglim} allows to
conclude the proof of the upper bound, and therefore the one of Theorem~\ref{th:main}.

\subsubsection*{Proof of Lemma~\ref{lem:boundremainder}.}
We follow the strategy of~\cite{gantert2014large}
by first noting that $\log y \leq y - 1$ for any $y>0$. Noting that the
exponential is increasing and using a Taylor expansion, we also get
$\e^y - 1 \leq y + y^2/2+....+\e^y y^{k+1}/(k+1)!$
where for now~$k$ is an arbitrary large integer. We thus obtain
\begin{equation}
  \label{eq:intbound}
n^{1-\alpha}  \log \E\left[\ind_{\{ X_1 < nx \}} \e^{\theta n^{\alpha-1}X_1}  \right] 
 \leq
n^{1-\alpha}\sum_{j=1}^k
\E\left[ \frac{(\theta n^{\alpha - 1} X_1)^j}{j!}\ind_{\{ X_1 < nx \}}\right] + \frac{R_n}{(k+1)!},
\end{equation}
where
\[
R_n = n^{1-\alpha}  (\theta n^{\alpha-1})^{k+1}
\E\left[ X_1^{k+1} \ind_{\{ X_1 < nx \}} \e^{\theta n^{\alpha - 1} X_1}
  \right].
\]
For the sum on the right hand side of~\eqref{eq:intbound}, the term
for $j=1$ is equal to zero because $\E[X_1]=0$. For $j>1$ each term is bounded by
\[
n^{1-\alpha}\E\left[|X_1|^j\right] (n^{\alpha-1})^j = n^{(\alpha - 1)(j-1)}\E\left[|X_1|^j\right],
\]
which goes to zero as $n\to + \infty$ since~$X_1$ has finite moments of any order
by Assumption~\ref{as:basics}.

It thus only remains to show that $\limsup\, R_n \leq 0$.
For this we use Holder's inequality for some $p,q>1$ with $1/p + 1/q = 1$
to separate the exponential and polynomial moment parts:
\begin{equation}
  \label{eq:Rnseparate}
R_n\leq n (\theta n^{\alpha-1})^{k+1} 
\E\left[ |X_1|^{(k+1)p} \ind_{\{ X_1 < nx \}}
  \right]^{1/p}\left(\frac{1}{n^\alpha}
\E\left[  \ind_{\{ X_1 < nx \}} \e^{q \theta n^{\alpha-1} X_1}
  \right]^{1/q}\right).
\end{equation}
For the first term we have
\begin{equation}
  \label{eq:Rnfirstterm}
n (\theta n^{\alpha-1})^{k+1} 
\E\left[ |X_1|^{(k+1)p} \ind_{\{ X_1 < nx \}}
  \right]^{1/p}\leq \theta^{k+1} n^{(\alpha-1)(k+1) +1}
\E\left[ |X_1|^{(k+1)p}   \right]^{1/p},
\end{equation}
which goes to zero for any $p> 1$ as soon as~$k$ is large enough for the following condition
to hold:
\begin{equation}
  \label{eq:condk}
  \alpha < \frac{k}{k+1}.
\end{equation}
Since $\alpha <1$, we can then choose~$k$ such that
\[
k > \frac{\alpha}{1-\alpha},
\]
in which case~\eqref{eq:condk} is satisfied and the right hand side
of~\eqref{eq:Rnfirstterm} goes to zero for any $p>1$.

The last step is to prove that there is some $q>1$ such that
the second term on the right hand side of~\eqref{eq:Rnseparate} satisfies
\begin{equation}
  \label{eq:lastboundupper}
  \underset{n\to +\infty}{\limsup}\
 \frac{1}{n^\alpha}
\E\left[  \ind_{\{ X_1 < nx \}} \e^{q \theta n^{\alpha - 1} X_1}
  \right]^{1/q} < +\infty.
\end{equation}
Lemma~\ref{lem:IBP} in Appendix~\ref{sec:appendix} implies that
(since the first boundary term is zero and the second is negative):
\[
\E\left[\ind_{\{ X_1 < nx \}} \e^{q \theta n^{\alpha - 1} X_1}
  \right] \leq q \theta n^{\alpha - 1}
\int_{-\infty}^{nx}
\e^{q\theta n^{\alpha - 1} z}\, \proba(X_1 \geq z)\,dz.
\]
We have
\[
\begin{aligned}
\int_{-\infty}^{nx}
\e^{q\theta n^{\alpha - 1} z}\, \proba(X_1 \geq z)\,dz & = \int_{-\infty}^{0}
\e^{q\theta n^{\alpha - 1} z}\, \proba(X_1 \geq z)\,dz
+ \int_{0}^{nx}
\e^{q\theta n^{\alpha - 1} z}\, \proba(X_1 \geq z)\,dz\\ & \leq
C_- + \int_{0}^{nx}
\e^{q\theta n^{\alpha - 1} z}\, \proba(X_1 \geq z)\,dz,
\end{aligned}
\]
where
\[
C_- = \int_{-\infty}^{0}
\e^{q\theta z}\,\,dz < +\infty.
\]
We therefore focus on the behavior of the integral on~$[0,nx]$ as $n\to+\infty$.
Using again Tchebychev's inequality at subexponential scale for some~$\eta\in(0, \xi)$ together with
the change of variable $z = nx y$ (recall $x>0$ is fixed) we have
\begin{equation}
  \label{eq:tchebz}
\int_{0}^{nx}
\e^{q\theta  n^{\alpha - 1}z}\, \proba(X_1 \geq z)\,dz \leq \int_{0}^{nx}
\e^{q\theta  n^{\alpha - 1}z - \eta z^\alpha + \lambda_\alpha(\eta)} \,dz = nx
\, \e^{ \lambda_\alpha(\eta)} \int_{0}^{1}
\e^{n^\alpha g(y)}\,dy,
\end{equation}
where we introduced the function~$g$ defined by
\begin{equation}
  \label{eq:gz}
  \forall\, y \in[0,1],\quad
  g(y) = q\theta x y - \eta x^\alpha y^{\alpha}.
\end{equation}
Recall that for now $\theta < \xi x^{\alpha - 1}$ is fixed. Therefore, for~$\eps>0$ small
enough, it holds
\begin{equation}
  \label{eq:ineqtheta}
  \theta x < (1- \eps)^2\xi x^{\alpha }.
\end{equation}
We can then choose~$q=1/(1-\eps)>1$ and~$\eta=(1-\eps)\xi<\xi$.
In this case~\eqref{eq:ineqtheta} becomes
\[
q\theta x < \eta x^{\alpha}.
\]
The above condition implies in particular that
\[
  \forall\, y \in [0, 1], \quad g(z)\leq 0,
\]
so
\[
\forall\, n \geq 1,\quad
\int_{0}^{1} \e^{n^\alpha g(y)}\,dy\leq 1.
\]
Finally, we gather the above estimates to reach
\[
\begin{aligned}
\frac{1}{n^\alpha}
\E\left[  \ind_{\{ X_1 < nx \}} \e^{q \frac{\beta_n}{n} X_1}
  \right]^{1/q} & \leq \frac{1}{n^\alpha}\left(
\frac{\theta n^{\alpha-1} }{1-\eps}
\left( C_- + n x\,\e^{ \lambda_\alpha(\eta)} \right)\right)^{1-\eps}\\
 & = \frac{1}{n^{\eps \alpha}}\left(
\frac{\theta}{1-\eps} \left( \frac{C_-}{n} + x\,\e^{ \lambda_\alpha(\eta)} \right)\right)^{1-\eps}
\xrightarrow[n\to+\infty]{}0.
\end{aligned}
\]
This shows that~\eqref{eq:lastboundupper} is satisfied, so
Lemma~\ref{lem:boundremainder} holds and the theorem is proved.

\section{Discussion}
\label{sec:discussion}
In this work we investigated large deviations principles for empirical averages of
i.i.d. subexponential random variables. Under a subexponential moment condition, we
showed that a LDP holds at a subexponential time scale with an explicit non convex rate function,
expressed through a tail coefficient.
This result generalizes earlier works by providing a full LDP that includes a lower bound,
and by avoiding assumptions on the cumulative distribution of the variable.

In this subexponential regime, the rate function is always singular at
zero, as it does not even admit a first derivative. Although we could expect non-existence
of a second derivative (because this would contradict the central limit theorem),
the phase transition from the standard exponential regime is very abrupt. Indeed, we illustrated in
Section~\ref{sec:examples} that a smooth rate function can become non-differentiable
by raising the underlying random variable to a power arbitrarily close to one. Moreover, in this new regime, the rate function does not depend
on the full probability distribution of the random variable but only on its coefficient in
the tail, so it independent of the light tail part of the distribution.
On the contrary the rate of decay in the number of variables is distribution specific,
while it is universally exponential for variables with exponential moments (although they may have
very different tails, like Gaussian and exponential variables).

Concerning the proof, although the upper bound part (which develops the techniques
used in~\cite{gantert2014large}) is interesting and helps understanding the problem, the most
original part of the paper is the proof of the lower bound. For this  we design
a new tilting strategy, which requires an assumption on the second derivative on the
free energy. This condition has the attractive interpretation of a control on the relative variance
of the unique tilted random variable.

An exciting outcome of this proof is to provide the optimal sampling strategy for
numerically estimating large deviations probabilities of subexponential variables.
It is generally believed that
a good tilting scheme in the heavy tail scenario is to replace~$X_1$ by $X_1 + nx$. We
show on the contrary that the optimal scheme is to replace~$X_1$ by the
variable~$\tilde X_{\eta_n}$ defined by~\eqref{eq:mutilde}-\eqref{eq:etan}.
As is usual for this kind of tilting, the optimal value~$\eta_n$ depends on
the inverse of the derivative of the free energy. However, in the exponential scaling
case, it does not depend on~$n$ and is applied to all variables. Here,  one variable only
is tilted with a parameter that depends on the full sample size.

Finally, our initial motivation to replace assumptions on the cumulative distribution by
a moment condition was to move forward to correlated systems instead of independent variables.
In particular, if
one studies empirical averages of stochastic differential equations in the long time limit,
it is hard to define a cumulative distribution to make an assumption on.
By harvesting the idea proposed in~\cite{bazhba2022large}, the author managed to
propose a simple extension of the present paper to the Ornstein--Uhlenbeck process raised
to an arbitrary power~\cite{ferre2024heavy}. We consider this as a first stone to complete~\cite{ferre2019large}
and propose a full understanding of fluctuations of time averages of SDEs.

\subsection*{Acknowledgments}
The author is particularly grateful towards Alain Rouault and Djalil Chafaï for
their advise to read papers written by Nina Gantert and collaborators.
The reference~\cite{gantert2014large}
was actually the most insightful I could find on this issue, and it was decisive
for my understanding of the upper bound.

\appendix
\section{A couple of technical results}
\label{sec:appendix}

We recall the definition of a large deviations principle.
\begin{definition}
  \label{def:LDP}
  A sequence of random variables~$(Z_n)_{n\geq 1}$ taking values in a topological
  space~$\cZ$ equipped
  with its Borel $\sigma$-field satisfies a large deviations principle at speed~$v_n$
  and with rate function~$I:\cZ\to[0,+\infty]$ if~$I$ is lower semicontinuous and for
  any measurable set~$B\subset \cZ$ it holds
  \[
  -\inf_{\mathring{B}}\, I \leq \underset{n\to+\infty}{\liminf}\ \frac{1}{v_n}\log
  \proba\big(Z_n\in B\big)
  \leq \underset{n\to+\infty}{\limsup}\, \frac{1}{v_n}\log
  \proba \big(Z_n\in B\big)\leq - \inf_{\overline{B}}\, I,
  \]
  where~$\mathring{B}$ and~$\overline{B}$ denote respectively the interior and the closure
  of~$B$ for the topology of~$\cZ$.
  Moreover we say that~$I$ is a good rate function if it has compact level sets, and
   that~$I$ is trivial if it is equal to~$0$ everywhere or equal to~$+\infty$
  everywhere except at~$\E[Z]$.
\end{definition}

We regularly use Tchebychev type inequalities in the paper. In our terminology,
the inequality at exponential scale is:
\[
\forall\, z,\eta \geq 0, \quad \proba (X\geq z)\leq \e^{-\eta z}\E\left[ \e^{\eta X}\right].
\]
In order to prove our results, we also need to scale the various quantities at hand
by~$\phi_\alpha$. We can then use Tchebychev's inequality at \emph{subexponential scale},
which reads
\[
\forall\, z,\eta \geq 0, \quad \proba (X\geq z)\leq \e^{-\eta z^\alpha}
\E\left[ \e^{\eta \phi_\alpha(X)}\right].
\]
It is a simple corollary of the first inequality.

For the proof of the upper bound, we recall the following useful integration
by part formula~\cite[Lemma~5]{gantert2014large}.
\begin{lemma}
  \label{lem:IBP}
  For any real-valued random variable~$X$ on a probability space~$(\Omega,\mathcal{F},\proba)$,
  for any $a>0$ and real numbers $r_1<r_2$, it holds
  \[
  \E\left[ \e^{a X}\ind_{\{r_1\leq X\leq r_2\}}
    \right]
  =
  a \int_{r1}^{r2} \e^{a z} \proba(X\geq z)\,dz
  +\e^{a r_1} \proba(X\geq r_1)
  -\e^{a r_2} \proba(X\geq r_2).
  \]
\end{lemma}

\bibliographystyle{abbrv}

\end{document}